\documentclass[4pt]{article} % use larger type; default would be 10pt

\usepackage[utf8]{inputenc} % set input encoding (not needed with XeLaTeX)
\usepackage{geometry} % to change the page dimensions
\usepackage{graphicx} % support the \includegraphics command and options

\usepackage[colorlinks,citecolor=red]{hyperref}%交叉引用
\usepackage{amsmath}
\usepackage{amsfonts}
\usepackage{dsfont}
\usepackage{mathrsfs}
\usepackage{amsthm}
\usepackage{latexsym}
\usepackage{graphicx}
\usepackage{amssymb}
\usepackage{float}%put figure at a reqired place
\usepackage{epsfig}
\usepackage{epstopdf}
\usepackage{color,xcolor}
\usepackage{booktabs} % for much better looking tables
\usepackage{array} % for better arrays (eg matrices) in maths
\usepackage{paralist} % very flexible & customisable lists (eg. enumerate/itemize, etc.)
\usepackage{verbatim} % adds environment for commenting out blocks of text & for better verbatim
\usepackage{subfig} % make it possible to include more than one captioned figure/table in a single float
\newtheorem{theorem}{Theorem}[section]
\newtheorem{lemma}[theorem]{Lemma}
\newtheorem{corollary}[theorem]{Corollary}
\newtheorem{proposition}[theorem]{Proposition}

\newtheorem{conjecture}[theorem]{Conjecture}

\usepackage{fancyhdr} % This should be set AFTER setting up the page geometry
\usepackage{sectsty}
\allsectionsfont{\sffamily\mdseries\upshape} % (See the fntguide.pdf for font help)
\usepackage[nottoc,notlof,notlot]{tocbibind} % Put the bibliography in the ToC
\usepackage[titles,subfigure]{tocloft} % Alter the style of the Table of Contents

\title{\textsf{The maximum spectral radius of irregular bipartite graphs}}

\author{{Jie Xue$^{a}$}, {Ruifang Liu$^{a}$}\thanks{Corresponding author. E-mail address: rfliu@zzu.edu.cn (Liu).},
{ Jiaxin Guo$^{a}$}, {Jinlong Shu$^{b}$}
\medskip
\\
{\footnotesize $^a$School of Mathematics and Statistics, Zhengzhou University, Zhengzhou, China}\\
{\footnotesize $^b$School of Computer Science and Technology, East China Normal University, Shanghai, China}}

\date{} % Activate to display a given date or no date (if empty),
\begin{document}
\maketitle

\begin{abstract}
\maketitle
A bipartite graph is subcubic if it is an irregular bipartite graph with maximum degree three.
In this paper, we prove that the asymptotic value of maximum spectral radius over subcubic bipartite graphs of order $n$ is
$3-\varTheta(\frac{\pi^{2}}{n^{2}})$. Our key approach is taking full advantage of the eigenvalues of certain tridiagonal matrices,
due to Willms [SIAM J. Matrix Anal. Appl. 30 (2008) 639--656].
Moreover, for large maximum degree, i.e., the maximum degree is at least $\lfloor n/2 \rfloor$,
we characterize irregular bipartite graphs with maximum spectral radius.
For general maximum degree, we present an upper bound on the spectral radius of irregular bipartite graphs in terms of the order and maximum degree.

\bigskip
\noindent {\bf AMS Classification:} 05C50

\noindent {\bf Key words:} Spectral radius, Bipartite graph, Irregular, Subcubic, Maximum degree
\end{abstract}

\section{Introduction}
The spectral radius of a graph is the largest eigenvalue of its adjacency matrix.
A classical issue in spectral graph theory is the Brualdi-Solheid problem \cite{Brualdi1986},
which aims to characterize graphs with extremal values of the spectral radius in a given class of graphs.
A lot of results concerning the Brualdi-Solheid problem were presented,
and some of these results were exhibited in a recent monograph on the spectral radius by Stevanovi\'{c} \cite{Stevanovic2018}.

Let $G$ be a connected graph on $n$ vertices with maximum degree $\Delta$.
The spectral radius of $G$ is denoted by $\rho(G)$, or simply $\rho$ when there is no scope for confusion.
By Perron-Frobenius theorem and Rayleigh quotient, it is easy to deduce a well-known upper bound $\rho(G)\leq \Delta$, with equality if and only if $G$ is regular.
This also yields that the spectral radius of a connected graph is strictly less than its maximum degree when the graph is irregular.
It is natural to ask what is the maximal value of the spectral radius of irregular graphs, which can be regarded as the Brualdi-Solheid problem for irregular graphs.
An equivalent statement of this problem is how small $\Delta-\rho$ can be when the graph $G$ is irregular.
A lower bound for $\Delta-\rho$ was first given by Stevanovi\'{c} in \cite{Stevanovic2004}.

\begin{theorem}{\rm(\cite{Stevanovic2004})}\label{th5}
  Let $G$ be a connected irregular graph of order $n$ and maximum degree $\Delta$. Then
  \begin{equation}
    \Delta-\rho(G)>\frac{1}{2n(n\Delta-1)\Delta^{2}}.
  \end{equation}
\end{theorem}
Some lower bounds for $\Delta-\rho(G)$ under other graph parameters,
such as the diameter and the minimum degree, were established in \cite{Cioaba2007EJC,Cioaba2007JCTB,Liu2007,Shi2009,Zhang2005}.
In particular, Cioab\u{a} \cite{Cioaba2007EJC} presented the following lower bound, which confirmed a conjecture in \cite{Cioaba2007JCTB}.
\begin{theorem}{\rm(\cite{Cioaba2007EJC})}
  Let $G$ be a connected irregular graph with $n$ vertices, maximum degree $\Delta$ and diameter $D$. Then
  \begin{equation}
    \Delta-\rho(G)>\frac{1}{nD}.
  \end{equation}
\end{theorem}

Another approach is to determine the asymptotic value of the maximum spectral radius for irregular graphs.
Denote by $\mathcal{F}(n,\Delta)$ the set of all connected irregular graphs on $n$ vertices with maximum degree $\Delta$.
Let $\rho(n,\Delta)$ be the maximum spectral radius of graphs in $\mathcal{F}(n,\Delta)$, that is,
$$\rho(n,\Delta)=\max\{\rho(G): G\in \mathcal{F}(n,\Delta)\}.$$
In \cite{Liu2007}, Liu, Shen and Wang proposed a conjecture for the asymptotic value of $\rho(n,\Delta)$.

\begin{conjecture}{\rm(\cite{Liu2007})}\label{conj1}
  For each fixed $\Delta$, the limit of $n^{2}(\Delta-\rho(n,\Delta))/(\Delta-1)$ exists. Furthermore,
  $$\lim_{n\rightarrow \infty}\frac{n^{2}(\Delta-\rho(n,\Delta))}{\Delta-1}=\pi^{2}.$$
\end{conjecture}

It is obvious that the conjecture holds for $\Delta=2$, since the path $P_{n}$ is the only graph in $\mathcal{F}(n,2)$, and its spectral radius is $2\cos(\frac{\pi}{n+1})$.
Very recently, the conjecture has been disproved by Liu \cite{Liu2022}. For subcubic graphs, it was proved that $\lim_{n\rightarrow \infty}n^{2}(3-\rho(n,3))=\pi^{2}/2$.
Moreover, the extremal graph with maximum spectral radius is path-like, and it is always non-bipartite.
Hence it is very interesting to consider the maximum spectral radius over irregular bipartite graphs.
The asymptotic value of the maximum spectral radius for subcubic bipartite graphs is determined.
Let us denote by $\mathcal{B}(n,\Delta)$ the set of all connected irregular bipartite graphs on $n$ vertices with maximum degree $\Delta$.
Thus, $\mathcal{B}(n,3)$ means the set of all connected subcubic bipartite graphs with $n$ vertices.
Let $\lambda(n,\Delta)$ be the maximum spectral radius among all graphs in $\mathcal{B}(n,\Delta)$, that is,
$$\lambda(n,\Delta)=\max\{\rho(G): G\in \mathcal{B}(n,\Delta)\}.$$
The first main result of the paper presents the limit of $n^{2}(\Delta-\lambda(n,\Delta))$ for $\Delta=3$.

\begin{theorem}\label{th1}
Let $\lambda=\lambda(n,3)$ be the maximum spectral radius among all graphs in $\mathcal{B}(n,3)$. Then
$$\lim_{n\to \infty}n^{2}(3-\lambda)=\pi^{2}.$$
\end{theorem}

Note that $\mathcal{B}(n,3)\subset\mathcal{F}(n,3)$.
This implies that $\lambda(n,3)<\rho(n,3)$. If we take $\Delta=3$,
then Conjecture \ref{conj1} means that
$$\lim_{n\rightarrow \infty}n^{2}(3-\rho(n,3))=2\pi^{2}.$$
However, according to Theorem \ref{th1}, it follows that
$$\lim_{n\rightarrow \infty}n^{2}(3-\rho(n,3))\leq \lim_{n\rightarrow \infty}n^{2}(3-\lambda(n,3))=\pi^{2},$$
which provides a counterexample to Conjecture \ref{conj1}.
A simple modification of Theorem \ref{th1} leads to the following consequence,
which establishes an asymptotic value of the maximum spectral radius for subcubic bipartite graphs.

\begin{theorem}
  The maximum spectral radius of subcubic bipartite graphs on $n$ vertices is $3-\varTheta(\frac{\pi^{2}}{n^{2}})$.
\end{theorem}

On the other hand, we focus on the maximum spectral radius of irregular bipartite graphs with large maximum degree.
Define an irregular bipartite graph $H_{n,\Delta}$ as follows.

\medskip

\noindent $\bullet$ For $2\Delta>n$, $H_{n,\Delta}$ is isomorphic to the complete bipartite graph $K_{\Delta,n-\Delta}$.\\
\noindent $\bullet$ For $2\Delta=n$, $H_{n,\Delta}$ is obtained from $K_{\Delta,\Delta}$ by deleting an edge.\\
\noindent $\bullet$ For $2\Delta=n-1$, $H_{n,\Delta}$ is obtained from $K_{\Delta,\Delta}$ by deleting an edge, and then adding a new vertex and joining it to a vertex of degree less than $\Delta$.

\medskip

When the maximum degree is at least $\lfloor n/2 \rfloor$, the maximum spectral radius is completely determined.

\begin{theorem}\label{th2}
  Let $G$ be an irregular bipartite graph on $n$ vertices with maximum degree $\Delta$.
  If $\Delta\geq \lfloor n/2 \rfloor$, then
  $$\rho(G)\leq \rho(H_{n,\Delta}),$$
  with equality if and only if $G\cong H_{n,\Delta}$.
\end{theorem}

Note that the above theorem also means that $\lambda(n,\Delta)=\rho(H_{n,\Delta})$ if $\Delta\geq \lfloor n/2 \rfloor$.
Furthermore, for general $n$ and $\Delta$, we provide a lower bound for $\Delta-\lambda(n,\Delta)$,
which is also an upper bound for $\lambda(n,\Delta)$.

\begin{theorem}\label{th3}
  Let $\lambda(n,\Delta)$ be the maximum spectral radius among all the graphs in $\mathcal{B}(n,\Delta)$. Then
  $$\Delta-\lambda(n,\Delta)>\frac{2\Delta}{n(4n+\Delta-4)}.$$
  \end{theorem}

  The rest of the paper is organized as follows. In Section 2, we prove Theorem \ref{th1} by utilizing the eigenvalue of certain tridiagonal matrices.
  The proof of Theorem \ref{th2} is presented in Section 3. In Section 4,
  we establish some structural properties of the irregular bipartite graph with maximum spectral radius, and then give the proof of Theorem \ref{th3}.
  In the final section, some additional remarks are provided.

\section{Subcubic bipartite graphs}

The eigenvalues and eigenvectors of a certain tridiagonal matrix were discussed in \cite{Willms2008}.
Let us consider the tridiagonal matrix
\begin{equation}
  A=\begin{bmatrix}
    -\alpha+b & c_{1} & \\
    a_{1} & b & c_{2} \\
      &  a_{2}  & \ddots  & \ddots \\
      &    &   \ddots  & b  & c_{n-1}\\
      &  &   &  a_{n-1}  &  -\beta+b
  \end{bmatrix}
\end{equation}
with the restriction $\sqrt{a_{i}c_{i}}=d\neq 0$ for $1\leq i\leq n-1$.
All variables appearing in the matrix are complex.
For some special cases, Willms \cite{Willms2008} presented the eigenvalues and corresponding eigenvectors.

\begin{lemma}{\rm(\cite{Willms2008})}\label{lem4}
  If $\alpha=d$ and $\beta=0$, then the eigenvalues of $A$ are
  $$\lambda_{i}=b+2d\cos\left(\frac{2i\pi}{2n+1}\right),$$
  where $1\leq i\leq n$.
\end{lemma}

We remark that the above lemma is one of the partial results in \cite{Willms2008}.
When $\alpha$ and $\beta$ take the other values, the eigenvalues and eigenvectors of $A$ were also provided.
Let us define a special tridiagonal matrix of order $n$:
\begin{equation}\label{eq3}
  M_{n}=\begin{bmatrix}
    1 & -1 & \\
    -1 & 2 & -1 \\
      &  -1  & \ddots  & \ddots \\
      &    &   \ddots  & 2  & -1\\
      &  &   &  -1  &  2
  \end{bmatrix}.
\end{equation}
Obviously, the matrix $M_{n}$ is obtained from $A$ by setting $\alpha=1$, $\beta=0$, $b=2$ and $a_{i}=c_{i}=-1$ for $1\leq i\leq n-1$.
Using Lemma \ref{lem4}, the eigenvalues and eigenvectors of $M_{n}$ can be determined directly.
In particular, we display its least eigenvalue in the following result.

\begin{proposition}\label{pro1}
  For a given positive integer $n$, the least eigenvalue of $M_{n}$ is $4\sin^{2}\left(\frac{\pi}{4n+2}\right)$.
\end{proposition}

Let us introduce an identical equation about the trigonometric function, which will be used in the sequel proof.
A fundamental result for the trigonometric function is
$$\cos^{2}\alpha+\cos^{2}(\frac{\pi}{2}-\alpha)=\sin^{2}\alpha+\sin^{2}(\frac{\pi}{2}-\alpha)=1.$$
Using this fact, for any integer $k\geq 3$, one can see that
\begin{equation}\label{eq6}
\sum_{i=1}^{k-1}\sin^{2}\left(\frac{i}{2k}\pi\right)=\frac{k-1}{2}
\end{equation}
and
\begin{equation}\label{eq7}
  \sum_{i=0}^{k-1}\cos^{2}\left(\frac{2i+1}{4k}\pi\right)=\frac{k}{2}.
\end{equation}

Given an irregular graph with maximum degree $\Delta$, we say that a vertex is unsaturated if its degree is less than $\Delta$.
The irregular graph with the maximum spectral radius cannot have many unsaturated vertices.
In \cite{Xue2022+}, Xue and Liu considered the subcubic bipartite graph with maximum spectral radius,
and showed that the extremal graph contains at most two unsaturated vertices.
Let $G$ be a bipartite graph with bipartition $(X,Y)$. We say that $G$ is balanced if $|X|=|Y|$.
For $n\geq 6$, let us define a series of bipartite graphs $B_{n}$ constructed as follows.

\medskip

\noindent $\bullet$ $B_{6}$ is a balanced bipartite graph obtained from $K_{3,3}$ by deleting an edge.\\
\noindent $\bullet$ If $B_{n}$ is balanced, then $B_{n+1}$ is obtained from $B_{n}$ by adding a vertex and joining it to an unsaturated vertex in one part of $B_{n}$.\\
\noindent $\bullet$ If $B_{n}$ is unbalanced, then $B_{n+1}$ is obtained from $B_{n}$ by adding a vertex and joining it to the unsaturated vertices of $B_{n}$.

\medskip

It was proved that $B_{n}$ is the unique subcubic bipartite graph with maximum spectral radius \cite{Xue2022+}.
The following lemma establishes the lower and upper bounds for the spectral radius of $B_{n}$.

\begin{figure}[t]
  \begin{center}
  \scalebox{0.85}[0.85]
  {\includegraphics{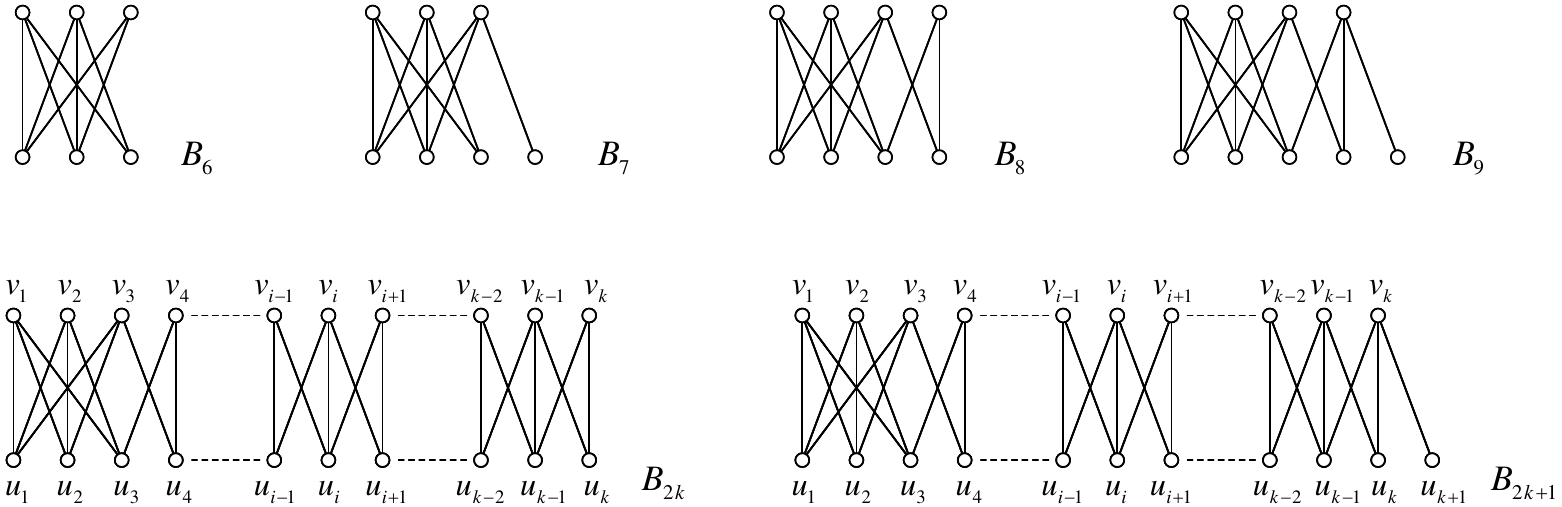}}
  \end{center}
  \caption{Subcubic bipartite graphs.}
  \label{fig1}
  \end{figure}

\begin{lemma}\label{lem5}
Let $\rho$ be the spectral radius of $B_{n}$ with $n\geq 6$. If $n$ is even, then
$$4\sin^{2}\left(\frac{\pi}{2n+2}\right)\leq 3-\rho\leq \frac{4n+48}{n-2}\sin^{2}\left(\frac{\pi}{2n}\right).$$
\end{lemma}
\begin{proof}
  Suppose that $n$ is even, that is, $n=2k$ with $k\geq 3$.
  We may label the vertices of $B_{2k}$ as presented in Fig. \ref{fig1}.
  Let $\mathtt{x}$ be the unit eigenvector of the adjacency matrix $A(B_{2k})$ corresponding to $\rho$.
  According to the symmetry of $B_{2k}$,
  one can see that $\mathtt{x}(u_{i})=\mathtt{x}(v_{i})$ for $1\leq i\leq k$.
  Set $\mathtt{x}(u_{i})=a_{i}$ for $1\leq i\leq k$.
  Let $L$ be the Laplacian matrix of $B_{2k}$,
  and $\Lambda$ be a diagonal matrix with diagonal elements $1,1,0,\ldots,0$.
  The two nonzero elements of $\Lambda$ are labeled by the vertices $u_{k}$ and $v_{k}$.
  Thus,
  \begin{equation}\label{eq1}
   3I-A(B_{2k})=L+\Lambda.
  \end{equation}
  Since $\rho=\mathtt{x}^{t}A(B_{2k})\mathtt{x}$, we have
  \begin{equation}\label{eq2}
    3-\rho=3\mathtt{x}^{t}\mathtt{x}-\mathtt{x}^{t}A(B_{2k})\mathtt{x}
    =\mathtt{x}^{t}(3I-A(B_{2k}))\mathtt{x}.
  \end{equation}
  Combining (\ref{eq1}) and (\ref{eq2}), it follows that $3-\rho=\mathtt{x}^{t}(L+\Lambda)\mathtt{x}$.
  Note that
  \[
    \mathtt{x}^{t} L \mathtt{x}=2(a_{1}-a_{3})^{2}+2 \sum_{i=1}^{k-1}(a_{i}-a_{i+1})^{2}~~~~\text{and}~~~\mathtt{x}^{t}\Lambda\mathtt{x}=2a_{k}^{2}.
  \]
  In $B_{2k}$, the vertices $u_{1}$ and $u_{2}$ have the same neighbors,
  hence $\mathtt{x}(u_{1})=\mathtt{x}(u_{2})$, that is, $a_{1}=a_{2}$. It follows that
  \begin{eqnarray}\label{eq4}
    \begin{split}
    3-\rho&=\mathtt{x}^{t}L\mathtt{x}+\mathtt{x}^{t}\Lambda\mathtt{x}\\
    &=2(a_{1}-a_{3})^{2}+2 \sum_{i=1}^{k-1}(a_{i}-a_{i+1})^{2}+2a_{k}^{2}\\
    &\geq 2 \sum_{i=1}^{k-1}(a_{i}-a_{i+1})^{2}+2a_{k}^{2}.
    \end{split}
  \end{eqnarray}
  Let $\mathtt{y}=(a_{1},a_{2},\ldots,a_{k})$.
  Suppose that $M_{k}$ is the tridiagonal matrix defined in (\ref{eq3}).
  Thus, one can see that
  \begin{equation}\label{eq5}
    \mathtt{y}^{t}M_{k}\mathtt{y}=a_{k}^{2}+\sum_{i=1}^{k-1}(a_{i}-a_{i+1})^{2}.
  \end{equation}
  Combining (\ref{eq4}) and (\ref{eq5}), we obtain that
  \begin{equation}
    3-\rho\geq 2\mathtt{y}^{t}M_{k}\mathtt{y}.
  \end{equation}
  Let $\lambda_{\min}$ be the least eigenvalue of the matrix $M_{k}$.
  Note that $\mathtt{y}^{t}\mathtt{y}=\sum_{i=1}^{k}a_{i}^{2}=(\mathtt{x}^{t}\mathtt{x})/2=1/2$.
  Therefore,
  $$\lambda_{\min}\leq \frac{\mathtt{y}^{t}M_{k}\mathtt{y}}{\mathtt{y}^{t}\mathtt{y}}=2\mathtt{y}^{t}M_{k}\mathtt{y}.$$
  By Proposition \ref{pro1}, we have $\lambda_{\min}=4\sin^{2}(\pi/(4k+2))$. Combining the above equations, it follows that
  \begin{equation}
    3-\rho\geq 2\mathtt{y}^{t}M_{k}\mathtt{y}\geq \lambda_{\min}=4\sin^{2}\left(\frac{\pi}{4k+2}\right)=4\sin^{2}\left(\frac{\pi}{2n+2}\right).
  \end{equation}

  In the following, we will give an upper bound for $3-\rho$. In order to prove the upper bound, we construct a vector on the vertices of $B_{2k}$.
  Let $\mathtt{z}$ be the $n$-vector whose entries satisfy
  $$\mathtt{z}(u_{i})=\mathtt{z}(v_{i})=\sin\left(\frac{k-i}{2k}\right),$$
  for $1\leq i\leq k$. According to (\ref{eq6}), we can see that
  \begin{equation}\label{eq8}
    \mathtt{z}^{t}\mathtt{z}=2\sum_{i=1}^{k}\sin^{2}\left(\frac{k-i}{2k}\pi\right)=2\sum_{i=1}^{k-1}\sin^{2}\left(\frac{i}{2k}\pi\right)=k-1.
  \end{equation}
  Note also that
  \begin{eqnarray}\label{eq9}
    \begin{split}
    \mathtt{z}^{t}(L+\Lambda)\mathtt{z}&=2(\mathtt{z}(u_{1})-\mathtt{z}(u_{3}))^{2}+2\mathtt{z}(u_{k})^{2}
        +2\sum_{i=1}^{k-1}(\mathtt{z}(u_{i})-\mathtt{z}(u_{i+1}))^{2}\\
    &=2\left(\sin\left(\frac{k-1}{2k}\pi\right)-\sin\left(\frac{k-3}{2k}\pi\right)\right)^{2}
        +2\sum_{i=0}^{k-2}\left(\sin\left(\frac{i}{2k}\pi\right)-\sin\left(\frac{i+1}{2k}\pi\right)\right)^{2}\\
    &=8\sin^{2}\left(\frac{2}{4k}\pi\right)\cos^{2}\left(\frac{2k-4}{4k}\pi\right)
        +8\sin^{2}\left(\frac{1}{4k}\pi\right)\sum_{i=0}^{k-2}\cos^{2}\left(\frac{2i+1}{4k}\pi\right)\\
    &\leq 24\sin^{2}\left(\frac{1}{4k}\pi\right)+8\sin^{2}\left(\frac{1}{4k}\pi\right)\sum_{i=0}^{k-1}\cos^{2}\left(\frac{2i+1}{4k}\pi\right)\\
    &=\left(4k+24\right)\sin^{2}\left(\frac{1}{4k}\pi\right),
    \end{split}
  \end{eqnarray}
  where the last equality is obtained from (\ref{eq7}).
  Moreover, since
  $$\rho\geq \frac{\mathtt{z}^{t}A(B_{2k})\mathtt{z}}{\mathtt{z}^{t}\mathtt{z}},$$
  we obtain that
  \begin{eqnarray}\label{eq10}
    3-\rho\leq \frac{3\mathtt{z}^{t}\mathtt{z}}{\mathtt{z}^{t}\mathtt{z}}-\frac{\mathtt{z}^{t}A(B_{2k})\mathtt{z}}{\mathtt{z}^{t}\mathtt{z}}
    =\frac{\mathtt{z}^{t}(L+\Lambda)\mathtt{z}}{\mathtt{z}^{t}\mathtt{z}}.
  \end{eqnarray}
  Combining (\ref{eq8}), (\ref{eq9}) and (\ref{eq10}), it follows that
  \begin{equation}
    3-\rho\leq \frac{\left(4k+24\right)\sin^{2}\left(\frac{\pi}{4k}\right)}{k-1}=\frac{4n+48}{n-2}\sin^{2}\left(\frac{\pi}{2n}\right),
  \end{equation}
  as required.
\end{proof}

Now, we are in a position to present the proof of  Theorem \ref{th1}.

\medskip
\noindent  \textbf{Proof of Theorem \ref{th1}.}
As mentioned above, $B_{n}$ is the unique bipartite graph in $\mathcal{B}(n,3)$ with the maximum spectral radius.
  Thus, $\lambda=\rho(B_{n})$. It is equivalent to show that
  $$\lim_{n\to \infty}n^{2}(3-\rho(B_{n}))=\pi^{2}.$$
  We divide the proof into two cases according to the parity of $n$.

  \smallskip

  \noindent{\bf Case 1}. $n$ is even.

  \smallskip

  According to Lemma \ref{lem5}, it follows that
  \begin{equation}
    \lim_{n\to \infty}4n^{2}\sin^{2}\left(\frac{\pi}{2n+2}\right)\leq \lim_{n\to \infty}n^{2}(3-\rho(B_{n}))\leq
    \lim_{n\to \infty}\frac{n^{2}(4n+48)}{n-2}\sin^{2}\left(\frac{\pi}{2n}\right).
  \end{equation}
  The left limit is
  $$\lim_{n\to \infty}4n^{2}\sin^{2}\left(\frac{\pi}{2n+2}\right)=\lim_{n\to \infty}\frac{4n^{2}\pi^{2}}{(2n+2)^{2}}=\pi^{2}.$$
  On the other hand, one can see that
  $$\lim_{n\to \infty}\frac{n^{2}(4n+48)}{n-2}\sin^{2}\left(\frac{\pi}{2n}\right)=\lim_{n\to \infty}\frac{n^{2}(4n+48)\pi^{2}}{4n^{2}(n-2)}=\pi^{2}.$$
  Thus, the result holds in this case.

  \smallskip

  \noindent{\bf Case 2}. $n$ is odd.

  \smallskip

  Obviously, $B_{n}$ is a proper subgraph of $B_{n+1}$, and $B_{n-1}$ is a proper subgraph of $B_{n}$.
  Using Lemma \ref{lem5} for $B_{n-1}$ and $B_{n+1}$, respectively, we obtain that
  $$3-\rho(B_{n-1})\leq \frac{4n+44}{n-3}\sin^{2}\left(\frac{\pi}{2n-2}\right)$$
  and
  $$3-\rho(B_{n+1})\geq 4\sin^{2}\left(\frac{\pi}{2n+4}\right).$$
  Note that $\rho(B_{n-1})<\rho(B_{n})< \rho(B_{n+1})$. This implies that
  $$\lim_{n\to \infty}n^{2}(3-\rho(B_{n+1}))\leq \lim_{n\to \infty}n^{2}(3-\rho(B_{n}))\leq \lim_{n\to \infty}n^{2}(3-\rho(B_{n-1})).$$
  It follows that
  $$\lim_{n\to \infty}4n^{2}\sin^{2}\left(\frac{\pi}{2n+4}\right)\leq \lim_{n\to \infty}n^{2}(3-\rho(B_{n}))
  \leq \lim_{n\to \infty}\frac{n^{2}(4n+44)}{n-3}\sin^{2}\left(\frac{\pi}{2n-2}\right),$$
  that is,
  $$\pi^{2}\leq \lim_{n\to \infty}n^{2}(3-\rho(B_{n})) \leq \pi^{2}.$$
  Hence $\lim_{n\to \infty}n^{2}(3-\rho(B_{n}))=\pi^{2}$. The proof is completed.
\hspace*{\fill}$\Box$

\section{Irregular bipartite graph with $\Delta\geq \left\lfloor n/2\right\rfloor$}

In this section, we consider the maximum spectral radius of irregular bipartite graphs
with $\Delta\geq \left\lfloor n/2 \right\rfloor$. An upper bound for the spectral radius
of bipartite graphs with given size was presented by Bhattacharya, Friedland and Peled \cite{Bhattacharya2008}.
The following lemma is part of \cite[Proposition 2.1]{Bhattacharya2008}.
\begin{lemma}{\rm(\cite{Bhattacharya2008})}\label{lem1}
  Let $G$ be a bipartite graph of size $m$. Then
  $$\rho(G)\leq \sqrt{m},$$
  with equality if and only if $G$ is a disjoint union of a complete bipartite graph and isolated vertices.
\end{lemma}

Note that the complete bipartite graph $K_{\Delta,n-\Delta}$ is irregular if $\Delta>\left\lfloor n/2\right\rfloor$.
Indeed, the maximum spectral radius can be attained for the complete bipartite graph $K_{\Delta,n-\Delta}$.

\begin{lemma}\label{lem2}
  Let $G$ be an irregular bipartite graph on $n$ vertices with maximum degree $\Delta$.
  If $\Delta>n-\Delta$, then
  $$\rho(G)\leq \rho(K_{\Delta,n-\Delta}),$$
  with equality if and only if $G\cong K_{\Delta,n-\Delta}$.
\end{lemma}
\begin{proof}
  Let $G$ be the irregular bipartite graph with maximum spectral radius.
  It suffices to show that $G\cong K_{\Delta,n-\Delta}$.
  Suppose that the bipartition of $G$ is $(X,Y)$. Without loos of generality, assume that
  $|X|\leq |Y|$. Since the maximum degree of $G$ is $\Delta$, $|Y|\geq \Delta$.
  If $|Y|=\Delta$, then $G$ is a spanning subgraph of $K_{\Delta,n-\Delta}$.
  Note that the spectral radius is strictly increasing when adding an edge in a graph.
  Thus, $G\cong K_{\Delta,n-\Delta}$.
  If $|Y|>\Delta$, then
  $$|E(G)|\leq |X|\times |Y|<\Delta(n-\Delta).$$
  By Lemma \ref{lem1}, we obtain that $\rho(G)\leq\sqrt{|E(G)|}<\sqrt{\Delta(n-\Delta)}$.
  However, the spectral radius of $K_{\Delta,n-\Delta}$ equals $\sqrt{\Delta(n-\Delta)}$,
  which implies that $\rho(K_{\Delta,n-\Delta})>\rho(G)$, a contradiction.
\end{proof}

The operation of edge transformation is a classic tool in spectral graph theory.
The following lemma appeared in a number of references (see, for example, \cite{Cvetkovic1997,Stevanovic2018,Wu2005}).

\begin{lemma}{\rm(\cite{Cvetkovic1997,Stevanovic2018,Wu2005})}\label{lem3}
  Let $u$ and $v$ be two vertices of a connected graph $G$, and let $S\subseteq N(u)\backslash N(v)$.
  Let $G'=G-\{wu:w\in S\}+\{wv:w\in S\}$.
  If $\mathtt{x}(v)\geq \mathtt{x}(u)$,
  where $\mathtt{x}$ is the principal eigenvector of $G$, then
  $\rho(G')>\rho(G)$.
\end{lemma}

In the following we present the proof of Theorem \ref{th2}. We remark that $H_{n,\Delta}$ is an irregular bipartite graph on $n$ vertices with maximum degree $\Delta$.
If $2\Delta>n$, then $H_{n,\Delta}$ contains $\Delta$ unsaturated vertices.
If $2\Delta=n$ or $n-1$, then $H_{n,\Delta}$ contains two unsaturated vertices.
Note also that $H_{6,3}\cong B_{6}$ and $H_{7,3}\cong B_{7}$ (see Fig. \ref{fig1}).

\medskip
\noindent  \textbf{Proof of Theorem \ref{th2}.}
Suppose that $G$ is the bipartite graph with maximum spectral radius. If $\Delta>\lfloor n/2 \rfloor$, then $\Delta>n-\Delta$.
  For this case, it follows from Lemma \ref{lem2} that $G\cong H_{n,\Delta}$.
  Suppose now that $\Delta=\lfloor n/2 \rfloor$.
  We may assume that the bipartition of $G$ is $(X,Y)$, where $|X|\leq |Y|$.
  According to the parity of $n$, we divide the proof into two cases.

  \smallskip

  \noindent{\bf Case 1}. $n$ is even, that is, $n=2k$ for some $k\geq 2$.

  \smallskip

  Here, $\Delta=k$. Thus, one can see that $|Y|\geq k$.
  Note that $K_{k,k-1}$ is a proper subgraph of $H_{n,\Delta}$. Thus, $\rho(H_{n,\Delta})>\rho(K_{k,k-1})=\sqrt{k(k-1)}$.
  If $|Y|\geq k+1$, then $|X|\leq k-1$. Hence, $|E(G)|\leq \Delta|X|\leq k(k-1)$.
  By Lemma \ref{lem1}, we obtain that $\rho(G)\leq \sqrt{k(k-1)}<\rho(H_{n,\Delta})$, a contradiction.
  Therefore, $|Y|=k$, which implies that $|X|=|Y|=k$. Obviously, $G$ is a proper subgraph of $K_{k,k}$.
  Moreover, since the spectral radius is strictly increasing when adding edges, then $G$ is the graph obtained from $K_{k,k}$ by deleting one edge,
  hence $G\cong H_{n,\Delta}$.

\smallskip

  \noindent{\bf Case 2}. $n$ is odd, that is, $n=2k+1$ for some $k\geq 2$.

  \smallskip

  In this case, $\Delta=k$. If $|Y|\geq k+2$, then $|X|=n-|Y|\leq k-1$. Thus, $|E(G)|\leq \Delta|X|\leq k(k-1)$.
  It follows from Lemma \ref{lem1} that $\rho(G)\leq \sqrt{k(k-1)}$.
  However, since $K_{k,k-1}$ is also a proper subgraph of $H_{n,\Delta}$,
  we have $\rho(H_{n,\Delta})>\rho(K_{k,k-1})=\sqrt{k(k-1)}$. Hence, $\rho(H_{n,\Delta})>\rho(G)$, a contradiction.
  Suppose that $|Y|\leq k+1$. This implies that $|X|=k$ and $|Y|=k+1$.
  Let $Y^{*}=\{v^{*}\in Y: d(v^{*})<k\}$. Thus, for any vertex $v\in Y\backslash Y^{*}$,
  we have $d(v)=k$, and hence $v$ is adjacent to all vertices of $X$.
  If $|Y\backslash Y^{*}|\geq k$, then the subgraph of $G$ induced by $X\cup (Y\backslash Y^{*})$ is isomorphic to $K_{k,|Y\backslash Y^{*}|}$.
  Thus, the maximum degree of $G$ is greater than $k$, a contradiction. Hence, $|Y\backslash Y^{*}|\leq k-1$, and so $|Y^{*}|\geq 2$.
  For vertices in $Y^{*}$, we obtain the following claim.

  \smallskip

  \noindent{\bf Claim}. If $u$ and $v$ are two vertices in $Y^{*}$, then $N(u)\cup N(v)=X$ and $N(u)\cap N(v)=\emptyset$.

  \smallskip

  If $X\backslash (N(u)\cup N(v))\neq \emptyset$, then we choose a vertex $w\in X\backslash (N(u)\cup N(v))$. Thus, $d(w)\leq |Y|-2=k-1$.
  Adding the edge $uw$, we have the resulting graph is also an irregular bipartite graph with maximum degree $\Delta$.
  However, the spectral radius of the resulting graph is greater than that of $G$, which contradicts the maximality of $G$.
  Therefore, we have $N(u)\cup N(v)=X$.
  Let $\mathtt{x}$ be the principal eigenvector of $G$. Without loos of generality, we may assume that $\mathtt{x}(u)\geq \mathtt{x}(v)$.
  Since $u\in Y^{*}$ and $N(u)\cup N(v)=X$, we can find a vertex $w\in N(v)\backslash N(u)$.
  If $N(u)\cap N(v)\neq\emptyset$, then the graph $G+wu-wv$ is also a connected irregular bipartite graph.
  However, according to Lemma \ref{lem3}, it follows that $\rho(G+wu-wv)>\rho(G)$, a contradiction.
  Therefore, $N(u)\cap N(v)=\emptyset$. The claim is proved.

  \smallskip

  Indeed, by Claim, it is easy to see that there are at most two vertices in $Y^{*}$. Combining the fact $|Y^{*}|\geq 2$,
  we know that $Y^{*}$ contains exactly two vertices, say $u$ and $v$. Suppose that $\mathtt{x}$ is the principal eigenvector of $G$,
  and $\mathtt{x}(u)\geq \mathtt{x}(v)$. Let $w^{*}$ be a vertex in $N(v)$. If $|N(v)|>1$, then we consider the graph
  $G'=G-\{wv:w\in N(v)\backslash \{w^{*}\}\}+\{wu:w\in N(v)\backslash \{w^{*}\}\}$.
  It follows from Lemma \ref{lem3} that $\rho(G')>\rho(G)$, a contradiction. Hence $|N(v)|=1$, that is, $v$ is a pendant vertex of $G$.
  Note also that the subgraph of $G$ induced by $X\cup (Y\backslash Y^{*})$ is isomorphic to $K_{k,k-1}$.
  Combining the above observations of $N(u)$ and $N(v)$, one can see that $G\cong H_{n,\Delta}$. Thus, we complete the proof.
\hspace*{\fill}$\Box$

\section{Lower bound of $\Delta-\lambda(n,\Delta)$ of irregular bipartite graphs}

In this section, we present a lower bound of $\Delta-\lambda(n,\Delta)$ for irregular bipartite graphs with general maximum degree $\Delta.$
As mentioned above, $\lambda(n,\Delta)$ is the maximum spectral radius of irregular bipartite graphs in $\mathcal{B}(n,\Delta)$.
Let $\Gamma$ be the irregular bipartite graph in $\mathcal{B}(n,\Delta)$, which attains the maximum spectral radius.
Hence $\rho(\Gamma)=\lambda(n,\Delta)$. In order to establish the lower bound of $\Delta-\lambda(n,\Delta)$,
it suffices to consider the lower bound of $\Delta-\rho(\Gamma)$.
Suppose that the bipartition of $\Gamma$ is $(X,Y)$. Let
$$X^{*}=\{w\in X: d(w)<\Delta\}~~~~ \text{and} ~~~~Y^{*}=\{w\in Y: d(w)<\Delta\}.$$
The following properties of $\Gamma$ are useful.

\begin{lemma}\label{lem6}
  The irregular bipartite graph $\Gamma$ satisfies the following properties.\\
  \noindent (I) $|X^{*}\cup Y^{*}|\geq 2$.\\
  \noindent (II) If $|X^{*}|\geq 1$, $|Y^{*}|\geq 1$ and $|X^{*}|+|Y^{*}|\geq 3$,
then the subgraph of $\Gamma$ induced by $X^{*}\cup Y^{*}$ is isomorphic to a complete bipartite graph.
\end{lemma}
\begin{proof}
  (I) Obviously, $|X^{*}\cup Y^{*}|\geq 1$ since $\Gamma$ is irregular. If $|X^{*}\cup Y^{*}|=1$,
  then $\Gamma$ contains exactly one vertex of degree less than $\Delta$. Thus, $\sum_{u\in X}d(u)\neq \sum_{v\in Y}d(v).$
  But, this is impossible, since $\Gamma$ is bipartite.

  (II) Suppose, to the contradiction, that there exist two nonadjacent vertices $w_{1}\in X^{*}$ and $w_{2}\in Y^{*}.$
Then we can obtain a new bipartite graph by adding the edge $w_{1}w_{2}.$
  Obviously, the new bipartite graph is also irregular, since $|X^{*}|+|Y^{*}|\geq 3$.
  However, the spectral radius of the new bipartite graph is greater than that of $\Gamma$, which contradicts the maximality of $\Gamma$.
\end{proof}

Let $\mathtt{x}$ be a unit eigenvector of $\Gamma$ corresponding to $\rho(\Gamma)$.
Define
 $$\mathtt{x}(\hat{w})=\max\{\mathtt{x}(w):w\in \Gamma\}~~~~\text{and}~~~~\mathtt{x}(\check{w})=\min\{x(w):w\in \Gamma\}.$$
Note that $\Gamma$ is irregular. One can see that $\hat{w}\neq \check{w}$.
Next we estimate the distance between $\hat{w}$ and $\check{w}$ in $\Gamma$.

\begin{lemma}\label{lem7}
   The distance between $\hat{w}$ and $\check{w}$ is at most $2(n-1)/\Delta$.
\end{lemma}
\begin{proof}
  Suppose that $d(\check{w})=\Delta$. Since $\mathtt{x}$ is the eigenvector corresponding to $\rho(\Gamma)$,
  we have $\rho(\Gamma)\mathtt{x}=A(\Gamma)\mathtt{x}$. It follows that
  $$\rho(\Gamma)\mathtt{x}(\check{w})=\sum_{v\in N(\check{w})}\mathtt{x}(v)\geq \Delta\mathtt{x}(\check{w}),$$
  which implies that $\rho(\Gamma)\geq \Delta$, a contradiction. Therefore, we obtain that $d(\check{w})<\Delta$.
  Denote by $\text{dist}(\hat{w},\check{w})$ the distance between $\hat{w}$ and $\check{w}$.

  \smallskip

  \noindent{\bf Case 1}. $\hat{w}$ and $\check{w}$ belong to the same part.

  \smallskip

  Assume that $\{\hat{w},\check{w}\}\subseteq X$.
  Let $P: \hat{w}=u_{0}v_{1}u_{1}\cdots v_{k}u_{k}=\check{w}$ be a shortest path from $\hat{w}$ to $\check{w}$.
  Thus, $V(P)\cap X=\{u_{0},u_{1},\ldots, u_{k}\}$ and $V(P)\cap Y=\{v_{1},v_{2}\ldots, u_{k}\}$.
  We claim that $|X^{*}\cap V(P)|\leq 3$. If not, we can find two vertices $u_{p}$ and $u_{q}$ in $X^{*}\cap V(P)$ with $1\leq p<q<k$.
  Without loos of generality, we may suppose that $\mathtt{x}(u_{p})\leq \mathtt{x}(u_{q})$. Since $P$ is a shortest path, $v_{p}$ and $u_{q}$ are nonadjacent.
  It is easy to see that the graph $\Gamma-u_{p}v_{p}+u_{q}v_{p}$ also belongs to $\mathcal{B}(n,\Delta)$.
  By Lemma \ref{lem3}, we have $\rho(\Gamma-u_{p}v_{p}+u_{q}v_{p})>\rho(\Gamma)$, a contradiction.
  Note that $P$ is a shortest path. One can see that $N(u_{i})\cap N(u_{j})=\emptyset$ for $|i-j|\geq 2$.
  It follows that
  \begin{equation}\label{eq11}
    \left|\bigcup_{u\in V(P)\cap X}N(u)\backslash V(P)\right|\geq \left\lceil \frac{k-2}{2}\right\rceil(\Delta-2).
  \end{equation}

  If $|Y^{*}\cap V(P)|\geq 2$,
  then it follows from Lemma \ref{lem6} that $u_{k}$ is adjacent to all vertices of $Y^{*}\cap V(P)$,
  contradicting the fact that $P$ is a shortest path. Therefore, $|Y^{*}\cap V(P)|\leq 1$.
  According to the shortest path $P$, it follows that $N(v_{i})\cap N(v_{j})=\emptyset$ for $|i-j|\geq 2$.
  Thus, we obtain that
  \begin{equation}\label{eq12}
  \left|\bigcup_{v\in V(P)\cap Y}N(v)\backslash V(P)\right|\geq \left\lceil \frac{k-1}{2}\right\rceil(\Delta-2).
  \end{equation}
  Combining (\ref{eq11}) and (\ref{eq12}), it follows that
  \begin{eqnarray*}
    n&\geq& |V(P)|+\left|\bigcup_{u\in V(P)\cap X}N(u)\backslash V(P)\right|+\left|\bigcup_{v\in V(P)\cap Y}N(v)\backslash V(P)\right|\\
    &\geq &2k+1+\left\lceil \frac{k-2}{2}\right\rceil(\Delta-2)+\left\lceil \frac{k-1}{2}\right\rceil(\Delta-2)\\
    &\geq& k\Delta+1,
  \end{eqnarray*}
  which implies that $k\leq (n-1)/\Delta$. Hence $\text{dist}(\hat{w},\check{w})=2k\leq 2(n-1)/\Delta$.

  \smallskip

  \noindent{\bf Case 2}. $\hat{w}$ and $\check{w}$ belong to different parts.

  \smallskip

  Assume that $\hat{w}\in X$ and $\check{w}\in Y$.
  Let $P: \hat{w}=v_{1}u_{1}\cdots v_{k}u_{k}=\check{w}$ be a shortest path from $\hat{w}$ to $\check{w}$.
  Thus, $V(P)\cap X=\{u_{1},u_{2},\ldots, u_{k}\}$ and $V(P)\cap Y=\{v_{1},v_{2}\ldots, u_{k}\}$.

  If $|V(P)\cap X^{*}|\geq 3$, then there exist two vertices $u_{p}$ and $u_{q}$ in $V(P)\cap X^{*}$ with $1\leq p<q<k$.
  We may assume that $\mathtt{x}(u_{p})\leq \mathtt{x}(u_{q})$. Since $P$ is a shortest path, $v_{p}$ and $u_{q}$ are nonadjacent.
  It is easy to see that the graph $\Gamma-u_{p}v_{p}+u_{q}v_{p}$ also belongs to $\mathcal{B}(n,\Delta)$.
  By Lemma \ref{lem3}, we have $\rho(\Gamma-u_{p}v_{p}+u_{q}v_{p})>\rho(\Gamma)$, a contradiction.
  Hence $|V(P)\cap X^{*}|\leq 2$.

  If $|V(P)\cap Y^{*}|\geq 2$, it follows from Lemma \ref{lem6} that $u_{k}$ is adjacent to all vertices in $Y^{*}\cap V(P)$.
Hence $P$ cannot be the shortest, a contradiction. Therefore, $|V(P)\cap Y^{*}|\leq 1$.

  Similar to Case 1, we still see that
  $$\left|\bigcup_{u\in V(P)\cap X}N(u)\backslash V(P)\right|\geq \left\lceil \frac{k-2}{2}\right\rceil(\Delta-2),$$
  and
  $$\left|\bigcup_{v\in V(P)\cap Y}N(v)\backslash V(P)\right|\geq \left\lceil \frac{k-1}{2}\right\rceil(\Delta-2).$$
  Combing the above inequalities, it follows that
  \begin{eqnarray*}
    n&\geq& |V(P)|+\left|\bigcup_{u\in V(P)\cap X}N(u)\backslash V(P)\right|+\left|\bigcup_{v\in V(P)\cap Y}N(v)\backslash V(P)\right|\\
    &\geq& 2k+\left\lceil \frac{k-2}{2}\right\rceil(\Delta-2)+\left\lceil \frac{k-1}{2}\right\rceil(\Delta-2)\\
    &\geq& k\Delta.
  \end{eqnarray*}
  Thus, we have $\text{dist}(\hat{w},\check{w})=2k-1\leq (2n-\Delta)/\Delta\leq 2(n-1)/\Delta$, as required.
\end{proof}

The next theorem presents a lower bound for $\Delta-\rho(\Gamma)$. Before proceeding with the proof,
let us recall an inequality proposed by Shi \cite[Lemma 1]{Shi2009}. If $a,b>0$, then
\begin{equation}\label{eq14}
  a(p-q)^{2}+bq^{2}\geq\frac{abp^{2}}{a+b},
\end{equation}
with equality if and only if $q=ap/(a+b)$.

\begin{theorem}\label{th4}
  The spectral radius $\rho(\Gamma)$ satisfies
  $$\Delta-\rho(\Gamma)\geq \frac{2\Delta}{n(4n+\Delta-4)}.$$
\end{theorem}
\begin{proof}
  Let $\mathtt{x}$, $\hat{w}$ and $\check{w}$ be defined as above.
  Since $\mathtt{x}$ is the eigenvector corresponding to $\rho(\Gamma)$, by the Rayleigh quotient, we have
  \begin{equation*}
    \rho(\Gamma)=2\sum_{uv\in E(\Gamma)}\mathtt{x}(u)\mathtt{x}(v).
  \end{equation*}
Similar to the arguements in the proof of Lemma \ref{lem5}, we obtain that
  \begin{eqnarray}\label{eq13}
    \begin{split}
    \Delta-\rho(\Gamma)&=\Delta \mathtt{x}^{t}\mathtt{x}-2\sum_{uv\in E(\Gamma)}\mathtt{x}(u)\mathtt{x}(v)\\
    &=\sum_{uv\in E(\Gamma)}(\mathtt{x}(u)-\mathtt{x}(v))^{2}+\sum_{v\in X^{*}\cup Y^{*}}(\Delta-d(v))\mathtt{x}(v)^{2}.
    \end{split}
  \end{eqnarray}
  Let $P$ be the shortest path of length $k$ between $\hat{w}$ and $\check{w}$.
  Using Cauchy-Schwarz inequality, it follows that
  \begin{eqnarray}
    \sum_{uv\in E(\Gamma)}(\mathtt{x}(u)-\mathtt{x}(v))^{2}\geq \sum_{uv\in E(P)}(\mathtt{x}(u)-\mathtt{x}(v))^{2}
    \geq \frac{1}{k}\left(\sum_{uv\in E(P)}(\mathtt{x}(u)-\mathtt{x}(v))\right)^{2}
    =\frac{1}{k}(\mathtt{x}(\hat{w})-\mathtt{x}(\check{w}))^{2}.
  \end{eqnarray}
  Recall that $\mathtt{x}(\hat{w})$ and $\mathtt{x}(\check{w})$ are the maximum and minimum components, respectively.
  One can see that
  $$\mathtt{x}(\check{w})^{2}<\frac{1}{n}<\mathtt{x}(\hat{w})^{2}.$$
  By Lemma \ref{lem6}, we have $|X^{*}\cup Y^{*}|\geq 2$, and hence
  \begin{eqnarray}
    \sum_{v\in X^{*}\cup Y^{*}}(\Delta-d(v))\mathtt{x}(v)^{2}\geq \mathtt{x}(\check{w})^{2}\sum_{v\in X^{*}\cup Y^{*}}(\Delta-d(v))
    \geq 2\mathtt{x}(\check{w})^{2}.
  \end{eqnarray}
  According to (\ref{eq14}), it follows that
  \begin{eqnarray}\label{eq15}
    \frac{1}{k}(\mathtt{x}(\hat{w})-\mathtt{x}(\check{w}))^{2}+2\mathtt{x}(\check{w})^{2}\geq \frac{2}{2k+1}\mathtt{x}(\hat{w})^{2}
    >\frac{2}{n(2k+1)}.
  \end{eqnarray}
 Combining Lemma \ref{lem7} and (\ref{eq13})-(\ref{eq15}), we have
  $$\Delta-\rho(\Gamma)>\frac{2}{n(2k+1)}\geq \frac{2\Delta}{n(4n+\Delta-4)},$$
  which completes the proof.
\end{proof}

By $\rho(\Gamma)=\lambda(n,\Delta)$, Theorem \ref{th3} follows from Theorem \ref{th4}.
It should be noted that the spectral radius of any irregular bipartite graph is not greater than that of $\Gamma$.
Hence Theorem \ref{th4} implies immediately the following consequence.

\begin{corollary}
  Let $G$ be a connected irregular bipartite graph on $n$ vertices with maximum degree $\Delta$. Then
  $$\Delta-\rho(G)>\frac{2\Delta}{n(4n+\Delta-4)}.$$
\end{corollary}

Obviously, for bipartite graphs, this bound is better than the bound in Theorem \ref{th5}, due to Stevanovi\'{c}.

\section{Concluding remarks}

The maximum spectral radius of irregular bipartite graphs is considered in this paper.
In order to get a better estimation of $\lambda(n,\Delta)$, we need to find the extremal graph in $\mathcal{B}(n,\Delta)$ with maximum spectral radius.
The extremal graph in $\mathcal{B}(n,3)$ is completely determined in \cite{Xue2022+}.
If $\Delta\geq \left\lfloor n/2\right\rfloor$, the extremal graph in $\mathcal{B}(n,\Delta)$ is determined in Theorem \ref{th2}.
For other cases, the extremal graph in $\mathcal{B}(n,\Delta)$ is still unknown.
Lemmas \ref{lem6} and \ref{lem7} present some structural properties of the extremal graph, which may be useful for finding the extremal graph in $\mathcal{B}(n,\Delta)$.

To determine the extremal graph, it is crucial to determine its degree sequence.
The degree sequence of the extremal graph in $\mathcal{F}(n,\Delta)$ was considered by Liu and Li \cite{Liu2008}.
They conjectured that the extremal graph in $\mathcal{F}(n,\Delta)$ has exactly one vertex of degree less than $\Delta$.
Until now, very little progress has been made on the degree sequence of the extremal graph in $\mathcal{F}(n,\Delta)$ (see \cite{Liu2009,Liu2012}).
Naturally, we focus on finding possible degree sequence of the extremal graph in $\mathcal{B}(n,\Delta)$.
We have seen from Lemma \ref{lem6} that the extremal graph in $\mathcal{B}(n,\Delta)$ has at least two vertices of degree less than $\Delta$.

There exists an analogous problem of finding the minimum algebraic connectivity of regular graphs.
The cubic graph with minimum algebraic connectivity was determined in \cite{Brand2007}.
Recently, Abdi, Ghorbani and Imrich \cite{Abdi2021} obtained the asymptotic value of the minimum algebraic connectivity of cubic graphs,
and presented the structure of the quartic graph with minimum algebraic connectivity.
For bipartite case, we investigated the minimum algebraic connectivity of cubic bipartite graphs. 
Moreover, the unique cubic bipartite graph with minimum algebraic connectivity was also completely characterized.
The structure of the extremal graph with minimum algebraic connectivity is very similar to the extremal graph $B_{n}$.
So we think that this is possibly an effective approach for finding the extremal graph with maximum spectral radius in $\mathcal{B}(n,\Delta)$.

\section*{Acknowledgements}

%The authors would like to thank the anonymous referees for their helpful comments on improving the presentation of the paper.
This work was supported by National Natural Science Foundation of China (Nos. 12001498 and 11971445)
and Natural Science Foundation of Henan Province (No. 202300410377).

%\section*{References}

\end{document}